\documentclass[english]{amsart}
\usepackage{ucs}
\usepackage[utf8x]{inputenc}
\usepackage{babel}
\usepackage{amsmath,amsthm}
\usepackage{amssymb}
\usepackage[all]{xy}
\usepackage{verbatim}

\newcommand{\R}{\mathbb{R}}
\newcommand{\C}{\mathbb{C}}
\newcommand{\Q}{\mathbb{Q}}
\newcommand{\Z}{\mathbb{Z}}
\newcommand{\N}{\mathbb{N}}

\newcommand{\ra}{\rightarrow}

\newtheorem{thm}{Theorem}
\newtheorem{df}[thm]{Definition}
\newtheorem{prop}[thm]{Proposition}
\newtheorem{lem}[thm]{Lemma}
\newtheorem{cor}[thm]{Corollary}

\title{The Tate conjecture for $K3$ surfaces over finite fields}

\author{Fran\c{c}ois Charles}
\email{francois.charles@univ-rennes1.fr}
\address{Universit\'e de Rennes 1, IRMAR--UMR 6625 du CNRS, Campus de Beaulieu, 35042 Rennes
Cedex, France}

\begin{document}

\begin{abstract}
Artin's conjecture states that supersingular $K3$ surfaces over finite fields have Picard number $22$. In this paper, we prove Artin's conjecture over fields of characteristic $p\geq 5$. This implies Tate's conjecture for $K3$ surfaces over finite fields of characteristic $p\geq 5$. Our results also yield the Tate conjecture for divisors on certain holomorphic symplectic varieties over finite fields, with some restrictions on the characteristic. As a consequence, we prove the Tate conjecture for cycles of codimension $2$ on cubic fourfolds over finite fields of characteristic $p\geq 5$. 
\end{abstract}

\maketitle

\section{Introduction}

The goal of this paper is to study the Tate conjecture for varieties with $h^{2,0}=1$ over finite fields. The main result is the following. Recall that Artin conjectured in \cite{Ar74} that the rank of the N\'eron-Severi group of a supersingular $K3$ surface over a finite field has the maximal possible value, that is, $22$.

\begin{thm}\label{main}
Artin's conjecture holds for supersingular $K3$ surfaces over algebraically closed fields of characteristic $p\geq 5$.
\end{thm}

Let $X$ be a smooth projective variety over a finite field $k$. Let $\ell$ be a prime number different from the characteristic of $k$. Tate conjectured in \cite{Ta65} that the Frobenius invariants of the space $H^{2i}(X_{\overline k}, \Q_{\ell}(i))$ are spanned by cohomology classes of algebraic cycles of codimension $i$. Using results of Nygaard-Ogus in \cite{NO85}, Theorem \ref{main} implies the following.

\begin{cor}
The Tate conjecture holds for $K3$ surfaces over finite fields of characteristic $p\geq 5$.
\end{cor}

As a consequence of the main theorem of \cite{LMS11}, this implies the following finiteness result.

\begin{cor}
Let $k$ be a finite field of characteristic $p\geq 5$. There are only finitely many $K3$ surfaces over $k$ up to isomorphism.
\end{cor}

With the extra assumption that $p$ is large enough with respect to the degree of a polarization of the $K3$ surface, Theorem \ref{main} is the main result of \cite{Mau12}. Our strategy uses that of \cite{Mau12} as a starting point. In particular, we also use, as a key geometric input, Borcherds' construction of automorphic forms for $O(2,n)$ \cite{Bo98, Bo99}, which allows one to find ample divisors supported on the Noether-Lefschetz locus for $K3$ surfaces.

A key point of Maulik's argument is to show that $K3$ surfaces have semistable reduction in equal positive characteristic. This is where the restrictions on the characteristic of the base field appear. Maulik then proceeds to showing that supersingular $K3$ are elliptic, which is enough to conclude that they satisfy the Tate conjecture by a result of Artin in \cite{Ar74}.

\bigskip

In this paper, we manage to circumvent the use of both semistable reduction for $K3$ surfaces and Artin's theorem on elliptic $K3$ surfaces, thus offering a direct proof of the Tate conjecture that gets rid of restrictions on the characteristic of the base field that appeared in \cite{Mau12}. These arguments allow us to prove the Tate conjecture for divisors on certain holomorphic symplectic varieties in any dimension, where showing semistable reduction seems out of reach at the moment, and where it is not clear what the analog of Artin's result might be. 

Recall that a complex irreducible holomorphic symplectic variety is a complex smooth, simply-connected variety $X$ such that $H^0(X, \Omega^2_X)$ is spanned by a unique holomorphic form that is everywhere non-degenerate. An important example is given by varieties of $K3^{[n]}$ type defined as the deformations of the Hilbert scheme of points on a $K3$ surface, see \cite{Be83}. The second singular cohomology group of a complex irreducible holomorphic symplectic variety is endowed with a canonical form called the Beauville-Bogomolov form, see \cite{Be83, Hu99}.

We deduce Theorem \ref{main} from the following result on the Tate conjecture for higher-dimensional varieties, with some restrictions on $p$.

\begin{thm}\label{weak}
Let $Y$ be a complex projective irreducible holomorphic symplectic variety of dimension $2n$ with second Betti number $b_2>3$. Let $h$ be the cohomology class of an ample line bundle on $Y$, let $d=h^{2n}$ and let $q$ be the Beauville-Bogomolov form.

Let $p\geq 5$ be a prime number. Assume that $p$ is prime to $d$ and that $p> 2n$. Suppose that $Y$ can be defined over a finite unramified extension of $\Q_p$ and that $Y$ has good reduction at $p$. Assume also that  $q$ induces a non-degenerate quadratic form on the reduction modulo $p$ of the primitive lattice in the second cohomology group of $Y$. Then the reduction $X$ of $Y$ at $p$ satisfies the Tate conjecture for divisors.
\end{thm}

In the case of varieties of $K3^{[n]}$ type, the assumptions of the theorem have the following explicit form.

\begin{cor}\label{K3n}
Let $Y$ be a complex polarized irreducible holomorphic symplectic variety of $K3^{[n]}$ type. Let $h$ be the cohomology class of an ample line bundle on $Y$, and let $d=q(h)$, where $q$ is the Beauville-Bogomolov form.

Let $p\geq 5$ be a prime number. Assume that $p$ is prime to $d$ and that $p> 2n$. Suppose that $Y$ can be defined over a finite unramified extension of $\Q_p$ and that $Y$ has good reduction at $p$. Then the reduction $X$ of $Y$ at $p$ satisfies the Tate conjecture.
\end{cor}

\noindent\textbf{Remark.} The assumption on $p$ ensures that the Hodge to de Rham spectral sequence of $X$ degenerates at $E_1$ by \cite{DI87}. 

\noindent\textbf{Remark.} For $K3$ surfaces, Theorem \ref{weak} is weaker than Theorem \ref{main}. However, proving Theorem \ref{weak} for fourfolds is a key step in removing the assumptions on the characteristic of the base field to get Theorem \ref{main}. We strongly expect that an extension of our method might relax the hypotheses on $p$ even in the higher-dimensional case.

\bigskip

Using the correspondence between cubic fourfolds and certain holomorphic symplectic varieties, we get the following instance of the Tate conjecture.

\begin{cor}\label{cubic}
Let $k$ be a finite field of characteristic $p\geq 5$, and let $X$ be a smooth cubic hypersurface in $\mathbb P^5_k$. Then $X$ satisfies the Tate conjecture for cycles of codimension $2$. 
\end{cor}

Note that the Tate conjecture for cubic fourfolds and divisors on holomorphic symplectic varieties over number fields was proved by Andr\'e in \cite{An96}. 

\bigskip

As in Andr\'e's work, we make heavy use of the Kuga-Satake correspondence between Hodge structures of $K3$ type and certain abelian varieties. We use this correspondence as well as general ideas on the deformation of cycle classes to prove the algebraicity of some cohomology classes. This type of argument is quite close in spirit to well-known results in Hodge theory around questions of algebraicity of Hodge loci. 

The main point, which appears in a slightly more involved form in Proposition \ref{inclus}, is that while the Kuga-Sataka correspondence is not known to be algebraic as the Hodge conjecture predicts, its existence is enough to provide mixed characteristic analogs of the Noether-Lefschetz loci, namely, universal deformation spaces for pairs $(X, \alpha)$, where $\alpha$ is a suitable Galois-invariant cohomology class. This allows one to study the lifting of such pairs to characteristic zero.

This method has the advantage of replacing degeneracy issues for family of holomorphic symplectic varieties by similar problems for abelian varieties, which are much easier to deal with. As a consequence, we do not use any of the birational arguments of \cite{Mau12}. Of course, these results are deep and beautiful in their own right. 

\bigskip

The plan of the paper is the following. We start by explaining how the results stated above can all be reduced to Theorem \ref{weak} for supersingular varieties. 

In section \ref{preliminary}, we gather some generalities around the deformation problems we deal with and recall some results of \cite{Mau12}. We state and prove them in the generality we need. 

In order to facilitate the exposition, we show in section \ref{proof1} that, with the notations of Theorem \ref{weak} when $X$ is supersingular, the Picard number of $X$ is at least $2$. We achieve this result by introducing a partial compactification of the Kuga-Satake mapping in mixed characteristic and using arguments related to the geometry of Hodge loci. Proposition \ref{inclus} contains the main geometric idea. 

In the last section of the paper, we prove Theorem \ref{weak} using the ideas of section \ref{proof1} and an induction process. Some of the lifting results there might be of independent interest. A surprising phenomenon is that the induction process does not allow us to directly show the Tate conjecture. However, we are able to use known cases of the Hodge conjecture for low-dimensional abelian varieties to conclude the proof.

\bigskip

We recently learned that Keerthi Madapusi Pera has announced results on the Tate conjecture for $K3$ surfaces. His proof seems to involve very different methods building on recent advances on the theory of canonical integral models of Shimura varieties.

\bigskip

\noindent\textbf{Acknowledgements.} The work on this paper started when I had the chance of discussing the paper \cite{Mau12} with Daniel Huybrechts during the workshop "Advances in hyperk\"ahler and holomorphic symplectic geometry" at the BIRS in Banff. 

I am very grateful to D. Huybrechts for the many discussions we had on this subject and his many remarks on early versions of this paper. I would also like to thank the organizers of the workshop for creating a stimulating atmosphere. I thank Davesh Maulik for helpful email correspondence, Olivier Benoist for interesting discussions and H\'el\`ene Esnault for helpful remarks on a first draft of this paper.

\section{Preliminary reductions}\label{reduction}

\subsection{Reduction to Theorem \ref{weak}}

In this section, we show how Theorem \ref{main}, Corollary \ref{K3n} and Corollary \ref{cubic} can be deduced from Theorem \ref{weak}.

\begin{proof}[Proof of Corollary \ref{K3n}]
Let $X$ and $Y$ be as in Corollary \ref{K3n}. We only need to show that $X$ and $Y$ satisfy the hypotheses of Theorem \ref{weak}, that is, that $p$ is prime to $h^{2n}$ and that  the Beauville-Bogomolov form $q$ induces a non-degenerate quadratic form on $H^2(Y, \Z)_{prim}\otimes \Z/p\Z$.

The second Betti number of $Y$ is either $22$ or $23$, hence it is strictly larger than $3$. Since $p>2n$ and $\frac{(2n)!}{n!2^n}q(h)^n=h^{2n}$, see for instance \cite[4.1.4]{OG08}, $p$ is prime to $h^{2n}$. Furthermore, the explicit description of the Beauville-Bogomolov form on the lattice $H^2(Y, \Z)$ as in \cite[Section 9]{Be83} shows that the $q$ induces a non-degenerate quadratic form on $H^2(Y, \Z)\otimes \Z/p\Z$. Since $p$ is prime to $q(h)$, $q$ induces a non-degenerate quadratic form on $H^2(Y, \Z)_{prim}\otimes \Z/p\Z$.

\end{proof}

\begin{proof}[Proof of Corollary \ref{cubic}]
Let $k$ be a finite field of characteristic at least $5$, and let $X$ be a smooth cubic hypersurface in $\mathbb P^5_k$. Let $F$ be the Fano variety of lines in $X$. It is a smooth projective variety of dimension $4$. 

Let $W$ be the ring of Witt vectors of $k$, and let $K$ be the fraction field of $W$. The hypersurface $X$ lifts to a cubic hypersurface $\mathcal X$ over $W$. Taking the Fano variety of lines gives a smooth lifting $\mathcal F$ of $F$ over $W$.

By results of Beauville-Donagi in \cite{BD85}, given an embedding of $K$ into $\C$, the variety $\mathcal F_{\C}$ is of $K3^{[2]}$ type. If $q$ is the Beauville-Bogomolov form, and $h$ is the ample class on $\mathcal F_{\C}$ is the ample class corresponding to the Pl\"ucker embedding, then $q(h)=6$. As a consequence, Corollary \ref{K3n} shows that $F$ satisfies the Tate conjecture for divisors.

\bigskip

Let $\ell$ be a prime number invertible in $k$. The incidence correspondence between $X$ and its variety of lines induces a morphism
$$\phi : H^4(X_{\overline k}, \Q_{\ell}(2))\ra H^2(F_{\overline k}, \Q_{\ell}(1))$$
that is equivariant with respect to the Frobenius action on both sides.

In \cite{BD85}, Beauville and Donagi show that the corresponding morphism over $\C$ induces an isomorphism between the primitive parts of the cohomology groups. By the smooth base change theorem, $\phi$ induces an isomorphism between the primitive parts of $H^4(X_{\overline k}, \Q_{\ell}(2))$ and $H^2(F_{\overline k}, \Q_{\ell}(1))$  as well. Consider the Poincar\'e dual of $\phi$
$$\psi : H^{6}(F_{\overline k}, \Q_{\ell}(3))_{prim}\ra H^4(X_{\overline k}, \Q_{\ell}(2))_{prim}.$$
It is also defined by an algebraic correspondence. In particular, it sends cohomology classes of $1$-cycles in $F_{\overline k}$ to classes of $2$-cycles in $X_{\overline k}$. 

By the hard Lefschetz theorem, and since $F$ satisfies the Tate conjecture for divisors, the group of Frobenius-invariant classes in $H^{6}(F_{\overline k}, \Q_{\ell}(3))$ is spanned by cohomology classes of $1$-cycles. This shows that the Frobenius-invariant part of $H^4(X_{\overline k}, \Q_{\ell}(2))$ is spanned by cohomology classes of codimension $2$ cycles and shows that cubic fourfolds satisfy the Tate conjecture.
\end{proof}

We now show how Corollary \ref{K3n} implies the Tate conjecture for $K3$ surfaces in any characteristic different from $2$ and $3$.

\begin{proof}[Proof of Theorem \ref{main}]
Let $S$ be a supersingular $K3$ surface over a finite field $k$ of characteristic at least $5$. Let $S^{[2]}$ be the Hilbert scheme that parametrizes length $2$ zero-dimensional subschemes of $S$. By \cite{Fo68}, $S^{[2]}$ is a smooth projective variety of dimension $4$. 

The variety $S^{[2]}$ is the quotient of the blow-up of $S\times S$ along the diagonal by the involution exchanging the two factors.  By \cite[Proposition 6]{Be83}, the second cohomology group of $S^{[2]}$ is generated by the second cohomology group of $S$ and the class $[E]$ of the exceptional divisor. As a consequence, $S$ satisfies the Tate conjecture if and only if $S^{[2]}$ satisfies the Tate conjecture for divisors.

\bigskip

By \cite{De81}, $S$ lifts to a projective $K3$ surface $\mathcal S$ over the ring $W$ of Witt vectors of $k$. The variety $S^{[2]}$ lifts to the relative Hilbert scheme $\mathcal S^{[2]}$. Let $K$ be the fraction field of $W$, and fix an embedding of $K$ into $\C$. Let $q$ be the Beauville-Bogomolov form on $H^2(\mathcal S_{\C}^{[2]}, \Z)$, and let $h$ be an ample cohomology class. Let $d=q(h)$.

For any large enough integer $N$, $Nh-[E]$ is an ample cohomology class. Since $q(E)=-2$, we have
$$q(Nh+E)=N^2d-2Nq(h,E)-2.$$
If $p$ divides $N$, then the variety $\mathcal S_{\C}^{[2]}$ with the polarization given by $Nh-[E]$ satisfies the hypothesis of Theorem \ref{weak}. This shows that $S$ satisfies the Tate conjecture.
\end{proof}

\noindent\textbf{Remark.} The idea of finding prime-to-$p$ polarizations on $S^{[2]}$ in order to study the Tate conjecture on $S$ is somewhat reminiscent of Zarhin's trick of finding a principal polarization on $(A\times\hat{A})^4$, where $A$ is an abelian variety, see \cite{Zar74}.

\noindent\textbf{Remark.} Note that the proof of Theorem \ref{main} only requires the special case of Theorem \ref{weak} in the supersingular case, that is, when the Galois action on the second cohomology group of the reduction of $X$ at $p$ is trivial.

\subsection{The universal deformation space and reduction of Theorem \ref{weak} to the supersingular case}\label{reduced}

Let us keep the notations of Theorem \ref{weak}. By the theorem of Deligne and Illusie in \cite{DI87}, the Hodge to de Rham spectral sequence of $X$ degenerates. By upper semicontinuity of cohomology groups, this implies that the Hodge numbers of $X$ and $Y$ are the same. Using the universal coefficients theorem, it is easy to check that the crystalline cohomology groups of $X$ are torsion free.

The versal formal deformation space of $X$ is smooth over the ring of Witt vectors $W$. Indeed, by the Bogomolov-Tian-Todorov theorem \cite{Bo78, Ti86, To89}, the versal deformation space of $Y$, that is, in characteristic zero, is smooth of dimension the dimension of $H^1(Y, T_Y)$. It follows that the versal deformation space of $X$ over $W$ has relative dimension at least the dimension of $H^1(Y, T_Y)\simeq H^1(Y, \Omega^1_Y)$, which is equal to the dimension of $H^1(X, \Omega^1_X)\simeq H^1(X, T_X)$ since the Hodge to de Rham spectral sequence degenerates at $E_1$. This implies that the versal formal deformation space of $X$ is smooth over $W$

As a consequence of these results, the deformation theory of $X$ is very similar to the deformation theory of $K3$ surfaces. In a more precise way, the second crystalline cohomology group of $X$ is a $K3$ crystal as in \cite{Og79}. The results of \cite{Og79}, paragraphs 1 and 2 on the versal deformation space of polarized $K3$ surfaces, as well as the results of \cite{Og11} hold without any change for the deformation of $X$. We will freely refer to these results.

\begin{df}
Let $X$ be as in Theorem \ref{weak}. We say that $X$ is supersingular if the Frobenius morphism acts on the second \'etale cohomology group of $X$ through a finite group. Otherwise, we say that $X$ has finite height.
\end{df}

There again, the general results of \cite{Ar74} apply and show that $X$ is supersingular if and only if the formal Brauer group of $X$ has finite height. These remarks show that the proof of \cite{NO85} gives without any change the following theorem.

\begin{thm}(Nygaard-Ogus, \cite{NO85}). Let $X$ be as in Theorem \ref{weak}. If $X$ has finite height, then $X$ satisfies the Tate conjecture.
\end{thm}
The supersingular case is thus the only remaining case of Theorem \ref{weak}. The proof of this case will be logically independent of the work of Nygaard-Ogus.

\section{Deformation theory and the Kuga-Satake morphism}\label{preliminary}

From now on, and through the remainder of this article, we will fix the following notations. Let $k$ be the algebraic closure of a finite field of characteristic $p\geq 5$. Let $W$ be the ring of Witt vectors of $k$, and let $K$ be the fraction field of $W$. By abuse of notation, we will again denote by $X$ the base change over $k$ of a variety satisfying the hypotheses of Theorem \ref{weak}.

We assume that $X$ is supersingular. Let $b=b_2(X)>3$ be the second Betti number of $X$ (which is equal to the second Betti number of $Y$ by the smooth base change theorem). We will show that the N\'eron-Severi group of $X$ has rank $b$. 

\bigskip

In this section, we gather results and notations around the deformation space of $X$, the Kuga-Satake mapping in that setting, and the Noether-Lefschetz locus. While some of these results are quite similar to those in \cite{Mau12}, and some are taken directly from there, we state them in our context and sometimes give different proofs and constructions. 

\subsection{Deformation spaces}\label{moduli}

Let $\widehat{\mathcal X}\ra \widehat S$ be the formal versal deformation space of $X$ over $W$. We showed in section \ref{reduced} that the assumptions on $X$ ensure that $\widehat S$ is formally smooth of relative dimension $b-2$ and that the deformation is universal. 

Let $L$ be the ample line bundle on $X$ induced by the ample line bundle on $Y$ with cohomology class $h$. Since $p$ is prime to $h^{2n}$, the class $c_1(L)^n$ is nonzero in $H^{2n}_{dR}(X/k)$, which in turn implies that $c_1(L)$ doesn't lie in $F^2H^2_{dR}(X/k)$, where $F^{\bullet}$ is the Hodge filtration on de Rham cohomology. 

Let $\widehat T$ be the universal deformation space of the pair $(X, L)$. By \cite[Proposition 2.3]{Og79}, $\widehat T$ is formally smooth of relative dimension $b-3$ over $W$. We also denote by $\widehat{\mathcal X}\ra \widehat T$ the universal formal deformation of the polarized variety $(X, L)$.

\bigskip

By Artin's algebraization theorem, we can find a smooth scheme $T$ of finite type over $W$, and a smooth projective morphism $\pi : \mathcal X\ra T$, together with a relatively ample line bundle $\mathcal L$ on $\mathcal X$ that extends the universal formal deformation of the pair $(X, L)$ over $\widehat T$. After shrinking $T$, we can assume that $\pi$ is a universal deformation at every point.

The Beauville-Bogomolov form on $Y$ is a rational multiple of the usual intersection form on the primitive cohomology lattice by \cite[p.775, Remarque 2]{Be83}. As a consequence, it descends to a quadratic form on the relative primitive cohomology over $T$, for the \'etale as well as for the crystalline theories. We will denote this extension by $q$ as well.

\subsection{Spin level structures}\label{spin}

We briefly recall basic definitions about spin level structures, see \cite{An96, Ri06, Mau12} to which we refer for further details. Let $n\geq 3$ be an integer, and assume $n$ is prime to $p$. 

Let $L$ be an abstract lattice that is isomorphic to the primitive cohomology lattice of $Y$ endowed with the Beauville-Bogomolov quadratic form. Let 
$$CSpin(L)=CSpin(L)(\mathbb A_f)\cap Cl_+(L\otimes\widehat{\mathbb Z}),$$
where $Cl_+(L)$ is the even part of the Clifford algebra of $L$.

Let $\mathbb K_n^{sp}$ be the subgroup of $CSpin(L)$ consisting of elements congruent to $1$ modulo $n$, and let $\mathbb K_n^{ad}$ be its image in $SO(L\otimes \widehat{\mathbb Z})$. If $\ell$ is a prime number, let $\mathbb K_{n, \ell}^{ad}$ be the $\ell$-adic part of $\mathbb K_n^{ad}$. By assumption, $q$ is a non-degenerate quadratic form on $L\otimes\Z/p\Z$ and by \cite[4.3]{An96}, $\mathbb K_{n, p}$ is the whole special orthogonal group.

\bigskip

We say that $\pi : \mathcal X\ra T$ admits a spin level $n$ structure if the image of the algebraic fundamental group of $T$ acting on primitive cohomology with $\ell$-adic coefficients lands into $\mathbb K_n^{ad}$ for any $\ell$ such that $\mathbb K_{n, \ell}^{ad}$ is a proper subgroup of the special orthogonal group, choosing a base point corresponding to the complex lift $Y$ of $X$.

After replacing $T$ by an \'etale cover, we can and will assume that $\pi : \mathcal X\ra T$ admits a spin level $n$ structure.

\subsection{The Kuga-Satake construction}\label{KS}

The Kuga-Satake mapping plays a major role in this paper. We refer to \cite{Mau12}, as well as to the papers of Andr\'e and Rizov \cite{An96, Ri10}, for most definitions and results. However, for reasons that will become clear below, we need to work with a slightly different definition as follows.

\bigskip

Let $V=H^2(Y, \Q)_{prim}$ be the primitive part of the second Betti cohomology group of the holomorphic symplectic variety $Y$, and let $C$ be the Clifford algebra of $V$. The classical Kuga-Satake construction, see \cite{De72} or \cite{CS11} for details, endows $C$ with a polarized Hodge structure of weight $1$. As a consequence, there exists a polarized abelian variety $A$ with $H^1(A, \Q)\simeq V$ as polarized Hodge structures. The integer lattice in $V$ determines $A$ uniquely. 

Let $g$ be the dimension of $A$ and $d'^2$ be the degree of the polarization. An explicit computation shows that $d'$ is prime to $p$, see \cite[5.1]{Mau12}. The polarized abelian variety $A$ is the \emph{Kuga-Satake variety} of $Y$.

Elements of $V$ act on $C$ by multiplication on the left. This induces a canonical primitive embedding of polarized Hodge structures
\begin{equation}\label{sub}
H^2(Y, \Z)_{prim}\hookrightarrow End(H^1(A, \Z)).
\end{equation}
Note that this canonical embedding only exists if we define the Kuga-Satake variety using the full Clifford algebra $C$ and not only its even part $C^+$ as in the references above. This is the reason we make this slight change in definition.

\subsection{The Kuga-Satake mapping}

We now proceed to the construction of a Kuga-Satake mapping over the deformation space $T$. First, let $\pi_K : \mathcal X_K\ra T_K$ be the generic fiber of $\pi$. Let $\mathcal A_{g, d', n}$ be the moduli space of abelian varieties of dimension $g$ with a polarization of degree $d'^2$ and a level $n$ structure over $W$, and let $\mathcal A_{g, d', n, K}$ be its generic fiber.

The following result is proved in \cite[Theorem 8.4.3]{An96}, see also \cite{Ri10} for the case of $K3$ surfaces. It follows from the fact that $\pi : \mathcal X\ra T$ admits a spin level $n$ structure. 

\begin{prop}
There exists a morphism
$$\kappa_K : T_K\ra \mathcal A_{g, d', n, K}$$
which for a complex point $t$ sends the variety $\mathcal X_t$ to its Kuga-Satake variety.

Given any prime number $\ell$, there is a canonical primitive embedding of $\ell$-adic sheaves on $T_K$
\begin{equation}\label{primitive}
R^2_{et}\pi_*\Z_{\ell}(1)_{prim}\hookrightarrow \mathrm{End}(R^1_{et}\psi_* \Z_{\ell}),
\end{equation}
where $R^2_{et}\pi_*\Z_{\ell}(1)_{prim}$ is the relative primitive cohomology of $\pi$ and $\psi : \mathcal A_K\ra T_K$ is the abelian scheme over $T$ induced by $\kappa_K$.
\end{prop}
\noindent\textbf{Remark.} The result of Andr\'e is actually stated for the usual intersection pairing on primitive cohomology. Since the Beauville-Bogomolov form $q$ is proportional to this intersection pairing, the same result holds using $q$.

We can now use Proposition 6.1.2 of \cite{Ri10} to conclude that the Kuga-Satake mapping extends to $T$ and get the following. 

\begin{prop}
The Kuga-Satake mapping $\kappa_K$ extends uniquely to a morphism 
$$\kappa : T\ra \mathcal A_{g, d', n}.$$
\end{prop}

\subsection{Quasi-finiteness of the Kuga-Satake mapping}

The following result is due to Maulik in the case $X$ is a $K3$ surface.

\begin{prop}[\cite{Mau12}, Proposition 5.10]\label{qf}
The Kuga-Satake map $\kappa  : T \ra \mathcal A_{g, d', n}$ is quasifinite.
\end{prop}

The proof of Maulik can easily be adapted to our setting. For the sake of completeness, let us however sketch a slightly more direct proof. We start with the following analog of equation (\ref{primitive}). We use the language of filtered Frobenius crystals as in \cite[Definition 6.3]{Mau12}.

\begin{prop}\label{crystalline}
Let $b$ be a $k$-point of $T$, and let $\widehat B$ be the formal neighborhood of $b$ in $T$. Denote by $\psi : \mathcal A\ra T$ the Kuga-Satake abelian scheme associated to $\mathcal X\ra T$. There is a canonical strict embedding of filtered Frobenius crystals on $\widehat B$
\begin{equation}\label{subcrys}
R^2\pi_*\Omega^{\bullet}_{\widehat{\mathcal X}/\widehat B}(1)_{prim} \hookrightarrow \mathrm{End}(R^1\psi_*\Omega^{\bullet}_{\widehat{\mathcal A}/\widehat B}).
\end{equation}
This morphism is compatible with Equation (\ref{primitive}) via the comparison theorems.
\end{prop}
\noindent\textbf{Remark.} Saying that the morphism above is strict means that the filtration on the left side is induced by the filtration on the right side.

\begin{proof}
Aside from the strictness property, this is proven in \cite[Section 6]{Mau12}. First, one argues that the morphism exists at the level of isocrystals by the comparison theorems of Andreatta-Iovita in \cite{AI12}. To check that the morphism is integral, one uses the theory of Fontaine-Messing in \cite{FM87}. Note that these arguments are general and do not use any property of the Beauville-Bogomolov form contrary to the subtler Morita arguments of \cite{Mau12}.

The morphism (\ref{subcrys}) is primitive because (\ref{primitive}) is. Strictness of (\ref{subcrys}) can be checked at the fibers, and is a general property of the theory of Fontaine-Messing, see for instance \cite[Proposition 3.1.1.1]{BM02}.
\end{proof}

\begin{proof}[Proof of Proposition \ref{qf}]
This is an easy consequence of Proposition \ref{crystalline} as in \cite[6.4]{Mau12}. 
\end{proof}

\subsection{Period maps}\label{period}

One of the main point of this paper is that a large part of \cite{Mau12} can be carried at the level of period spaces. We briefly gather some results on period maps for families of holomorphic symplectic varieties. Andr\'e's paper \cite{An96} contains related results.

\bigskip

Let $V$ be  vector space over $\Q$ of dimension $b-1$, and let $\psi$ be a non-degenerate bilinear form on $V$ of signature $(2, b-3)$ on $V$. Let $G$ be the algebraic group $SO(V, \psi)$, and let $\Omega$ be the period domain, that is, 
$$\Omega=\{\omega\in\mathbb P(V_{\C}) , \psi(\omega, \omega)=0 \,\mathrm{and}\, \psi(\omega, \overline\omega)>0\}.$$

To any $\omega\in\Omega$, we can associate a polarized Hodge structure of weight $2$ on $V$ such that $F^2V_{\C}=\C\omega$.As a consequence, the period domain can be naturally identified with a conjugacy class of morphisms $h : \mathbb S\ra G_{\R}$, where $\mathbb S$ is the Deligne torus. The pair $(G, \Omega)$ is a Shimura datum with reflex field $\Q$.

Now let $(L, \psi)$ be the lattice of section \ref{spin}, that is, a lattice isomorphic to the primitive lattice of $Y$. We consider the Shimura datum above associated to $V=L\otimes \Q$.

\bigskip

Let $n\geq 3$ be as before, and let $S_{n}$ be the Shimura variety defined over $\Q$ such that
$$\mathcal S_{n, \C}=G(\Q) \backslash  \Omega\times G(\mathbb A_f)/\mathbb K_n^{ad}.$$

Fix an embedding of $K$ into $\C$. Since $\mathcal X\ra T_K$ admits a level $n$ spin structure, the classical period map takes the form of an \'etale morphism of quasi-projective varieties
$$j : T_{\C}\ra \mathcal S_{n, \C}.$$

The  local Torelli theorem for holomorphic symplectic varieties in \cite[Th\'eor\`eme 5]{Be83} implies the following result.

\begin{prop}
The map
$$j : T_{\C}\ra \mathcal S_{n, \C}$$
is \'etale.
\end{prop}

The Kuga-Satake construction actually defines a morphism of Shimura varieties 
$$KS_{\C} : \mathcal S_{n, \C}\ra \mathcal A_{g, d', n, \C}.$$
It is a finite morphism. We get the following decomposition of the Kuga-Satake mapping.

\begin{prop}\label{factor}
The Kuga-Satake mapping
$$\kappa_{\C} : T_{\C}\ra \mathcal A_{g, d', n, \C}$$
factorizes as $KS_{\C}\circ j$, where
$$j : T_{\C}\ra \mathcal S_{n, \C}$$
is \'etale and 
$$KS_{\C} : \mathcal S_{n, \C}\ra \mathcal A_{g, d', n, \C}$$
is finite.
\end{prop}

\subsection{Divisors on the period space}\label{NL}

In this section, we recall a slightly adapted version of Theorem 3.1 of \cite{Mau12} that proves the ampleness of some components of the Noether-Lefschetz locus in the moduli space of polarized $K3$ surfaces. 

\bigskip

Recall that $d=q(h)$, where $h$ is the class of the polarization of $\mathcal X\ra T$. Let $L_0$ be a lattice containing $L$ such that the embedding $L\subset L_0$ is isomorphic to the embedding of the primitive cohomology of $Y$ into its full second cohomology group. We also denote by $h\in L$ the image of the ample class of $Y$.

If $\Lambda$ is a rank $2$ lattice of the form
$$\Lambda=\left(\begin{matrix} d & a\\
                                          a  & b\end{matrix}\right)$$
with negative discriminant, let $H_{\Lambda}$ be the locus in the period domain $\Omega$ of points $\omega$ such that there exists $v\in L$ with $\psi(v,v)=b$, $\psi(v, h)=a$ and $\psi(\omega, v)=0$.

By definition of $\mathbb K_n^{ad}$, the divisor $H_{\Lambda}$ descends to a divisor $D_{\Lambda}$ in $S_{n, \C}$. We will also denote by $D_{\Lambda}$ the divisor on the generic fiber of $T$ obtained via the period map.

\bigskip

Let $\lambda$ be the Hodge bundle on $\mathcal S_{n, \C}$. By definition, it is induced by the tautological line bundle over the period space. The Hodge bundle pulls back by the period map to the Hodge bundle on $T_{\C}$. Recall that the Hodge bundle on $T$, which we denote by $\lambda$ as well, is defined as
$$\lambda = \pi_* \Omega^2_{\mathcal X/T}.$$

The following Theorem is stated in \cite{Mau12} for period spaces of $K3$ surfaces, but the proof extends to the general case without any change. It relies in an essential way on Borcherds' results in \cite{Bo98, Bo99}.

\begin{thm}[\cite{Mau12}, Theorem 3.1]\label{borcherds}
Let $(\Lambda_k)_{k\in\N}$ be an infinite sequence of pairwise non isomorphic lattices as above. Then there exists a Cartier divisor $D$ on $\mathcal S_{n, \C}$, supported on a finite union of the $D_{\Lambda_k}$ such that 
$$\mathcal O(D)=\lambda^{\otimes a}$$
for some positive integer $a$.
\end{thm}

\section{Partial compactifications of the moduli space and existence of one line bundle}\label{proof1}

\subsection{Making the Kuga-Satake mapping finite}

One of the main results of \cite{Mau12}, and one that we are wishing to avoid, is the fact that families of supersingular $K3$ surfaces with semi-stable reduction do not degenerate. The analogous result for supersingular abelian varieties is well-known, see \cite[Proof of Theorem 1.1.a]{Oo74}, essentially because of the criterion of N\'eron-Ogg-Shafarevich. As a consequence, the result for $K3$ surfaces, or more generally for varieties as in Theorem \ref{weak}, would follow if the Kuga-Satake mapping were finite.

In this section, we give a very simple construction of a canonical partial compactification of $T$ over which the Kuga-Satake mapping extends to a finite morphism to $\mathcal A_{g, d', n}$.
\bigskip

As in Proposition \ref{factor}, the Kuga-Satake map $\kappa :T\ra \mathcal A_{g, d', n}$ admits a factorization over $\C$ through
$$j : T_{\C}\ra \mathcal S_{n, \C}$$
and 
$$KS_{\C} : \mathcal S_{n, \C}\ra \mathcal A_{g, d', n, \C}.$$

By Zariski's main theorem, there exists a normal variety $\widetilde{T}_{\C}$ over $\C$ containing $T_{\C}$ as an open subvariety such that $j$ extends to a finite morphism $\widetilde{T}_{\C}\ra \mathcal S_{n, \C}$. Since the Kuga-Satake map $T_{\C}\ra \mathcal A_{g, d', n, \C}$ is defined over $K$, it is easy to see that $\widetilde{T}_{\C}$ has a model $\widetilde{T}_K$ over $K$ such that the induced map 
$$\widetilde{T}_{\C}\ra \mathcal A_{g, d', n, \C}$$
is defined over $K$.

Let $T'$ be the normal scheme over $W$ defined by gluing the $W$-schemes $T$ and $\widetilde{T}_K$ along their common open subscheme $T_K$. By definition, the Kuga-Satake map extends to a morphism
$$\kappa' : T'\ra \mathcal A_{g, d', n}.$$
Since $\kappa$ is quasifinite by Proposition \ref{qf}, $\kappa'$ is quasifinite as well. The map $KS_{\C}$ above is a morphism of Shimura varieties. As a consequence, it is finite. This proves that $\kappa'$ is a finite morphism when restricted to the generic fiber of $T'$.

\bigskip

Now applying Zariski's main theorem, we can find a normal $W$-scheme $\overline T$ and a dominant open immersion $T'\hookrightarrow \overline T$ such that $$\kappa' : T'\ra \mathcal A_{g, d', n}$$ extends to a finite morphism 
$$\overline{\kappa} : \overline T \ra \mathcal A_{g, d', n}.$$

We can summarize the preceding construction in the following statement.

\begin{prop}\label{finite}
There exists a normal, separated $W$-scheme $\overline T$, and a dominant open immersion 
$$i : T\hookrightarrow \overline T$$
such that 
\begin{enumerate}
\item The Kuga-Satake map $\kappa$ extends to a finite morphism 
$$\overline{\kappa} : \overline T \ra \mathcal A_{g, d', n}.$$
\item The generic fiber of $\overline\kappa\circ i$ factorizes through the period map
$$j : T_{\C} \ra \mathcal S_{n, \C}.$$
\end{enumerate}
\end{prop}
The second condition above shows that the complex points of $\overline T$ parametrize Hodge structures of weight $2$.

\noindent\textbf{Remark.} Through the Kuga-Satake map, it should actually possible to give a modular interpretation of the space $\overline T$ and of its special fiber. It is for instance likely that the $p$-divisible group associated to the formal Brauer group of the universal variety over $T_k$ actually extends to a $p$-divisible group over $\overline T_k$. However, one of the points of this paper is that this modular interpretation is not needed.

\subsection{The supersingular locus in $\overline T$}

We start by defining the supersingular locus in $\overline T$.

\begin{df}
The supersingular locus in $\overline T$ is the inverse image by $\overline \kappa$ of the locus of supersingular abelian varieties in $\mathcal A_{g, d', n}$.
\end{df}

The following is one of the main points of our proof. It is a straightforward consequence of a result of Oort.

\begin{prop}\label{proper}
The supersingular locus in $\overline T$ is projective.
\end{prop}

\begin{proof}
In the course of the proof of Theorem 1.1.a in \cite{Oo74}, Oort proves that the locus of supersingular abelian varieties in $\mathcal A_{g, d', n}$ is a projective subvariety of $\mathcal A_{g, d', n}$. Since $\overline{\kappa}$ is finite, the supersingular locus in $\overline T$ is projective.
\end{proof}

Over $T$, the supersingular locus above coincides with the locus of points $t$ such that the fiber $\mathcal X_t$ is supersingular, as the next proposition shows. The corresponding result in the ordinary case was proved by Nygaard in \cite[Proposition 2.5]{Ny83}.

\begin{prop}\label{equal}
Let $X_t$ be a fiber of $\pi$ over a $k$-point of $T$. Then $X_t$ is supersingular if and only if the Kuga-Satake abelian variety of $X_t$ is supersingular.
\end{prop}

\begin{proof}
Let $\ell$ be a prime number prime to the characteristic of $k$. The varieties $X_t$ and $A$ are defined over a finite field $k_0$. We denote by $G_{k_0}$ the absolute Galois group of $k_0$. Let $P=H^2(X_t, \Q_{\ell}(1))_{prim}$ and let $C(P)$ be the Clifford algebra associated to $P$. By standard properties of the Kuga-Satake construction, there is an isomorphism of $G_{k_0}$-modules 
\begin{equation}\label{trivial}
C(P) \simeq \mathrm{End}_{C(P)}(H^1(A, \Q_{\ell})).
\end{equation}
We have to show that $X$ is supersingular if and only if $A$ is.

\bigskip 

Let us first assume that the dimension of $P$ is even. By \cite[Paragraphe 9, n. 4, Corollaire after Th\'eor\`eme 2]{BourA9}, $C(P)$ is a central simple algebra. Assume that $A$ is supersingular. Up to replacing $k_0$ by a finite extension, we can assume that the Frobenius acts trivially on $H^1(A, \Q_{\ell})$. Equation (\ref{trivial}) then shows that the Frobenius action on $P$ is trivial, which implies that $X_t$ is supersingular.

Conversely, if $X_t$ is supersingular, we can assume that the Frobenius morphism acts trivially on $\mathrm{End}_{C(P)}(H^1(A, \Q_{\ell}))$. As a consequence, it acts through the center of $C(P)$. Since this center is trivial, Frobenius acts on $H^1(A, \Q_{\ell})$ by a homothety. This implies that $A$ is supersingular.

\bigskip

In case the dimension of $P$ is odd, the even part $C^+(P)$ of the Cliffford algebra $C(P)$ is a central simple algebra by \cite[Paragraphe 9, n. 4, Th\'eor\`eme 3]{BourA9}. By standard properties of the Kuga-Satake construction, and up to replacing $k_0$ by a finite extension, $A$ is isogenous to the square of an abelian variety $B$ over $k_0$, such that there is an isomorphism of $G_{k_0}$-modules 
$$C^+(P) \simeq \mathrm{End}_{C^+(P)}(H^1(B, \Q_{\ell})).$$
Here $B$ is the Kuga-Satake variety used in \cite{Mau12}. We have to show that $X$ is supersingular if and only if $B$ is. The same argument as above shows this equivalence.
\end{proof}

\subsection{The closure of some Hodge loci}

The main result of this section is the key to avoiding the degeneration results of \cite{Mau12}. We investigate the geometrical properties of the Zariski closure of the Hodge locus $D_{\Lambda}$ of \ref{NL} in $\overline T$.

\bigskip

Before stating the result, we introduce the following notation. If 
$$\Lambda =\left(\begin{matrix}d & a\\
                                          a  & b\end{matrix}\right)$$  
is a lattice and $r$ is a positive integer, let 
$$\Lambda_r=\left(\begin{matrix}d & p^r a\\
                                          p^r a  & p^{2r} b\end{matrix}\right).$$
Note that there is an embedding of lattices $\Lambda_r\hookrightarrow\Lambda$ that sends the second base vector $v$ to $p^r v$. As a consequence, with the notations of \ref{NL}, the divisor $D_{\Lambda_r}$ contains the divisor $D_{\Lambda}$.

Recall that by construction the family $\pi : \mathcal X\ra T$ contains $X$ as a special fiber.

\begin{prop}\label{inclus}
Let $\Lambda$ be a lattice as in \ref{NL}, and let $\overline{D}_{\Lambda}$ be the Zariski closure of $D_{\Lambda}$ in $\overline T$. Let $C$ be the connected component of the supersingular locus of $\overline T$ passing through the point of $T$ corresponding to $X$.

If the intersection of $C$ and $\overline{D}_{\Lambda}$ is nonempty, then there exists a positive integer $r$ such that $\overline{D}_{\Lambda_r}$ contains the support of $C$.
\end{prop}

\noindent\textbf{Remark.} By the second part of Proposition \ref{finite}, we can view $D_{\Lambda}$ as a divisor on the generic fiber of $\overline T$, which is why the statement above makes sense.

\begin{proof}
The proposition above is close in spirit to the statement of \cite[Theorem 1.1.a]{Ar74}, and our point of view is somewhat similar to that of \cite{dJ12}. The main idea of our argument is very simple and can be summarized as follows. 

Assume that $C$ actually parametrizes a family of supersingular polarized varieties -- this is not true as $C$ does not need to lie in $T$. Then the assumption of the proposition means that there exists a polarized supersingular deformation $X_0$ of $X$, together with a line bundle $\mathcal L_0$  on $X_0$ such that the lattice generated by $\mathcal L_0$ and the polarization in the N\'eron-Severi group of $X_0$ is isomorphic to $\Lambda$. 

Now consider the versal deformation space of $(X_0, \mathcal L_0)$ inside $T$. The arguments from \cite[Theorem 1.6]{De81} show that it is a divisor $D$ in $T$ which is flat over $W$. By the argument of Artin above, see also \cite[Theorem 1]{dJ12}, and up to replacing $\mathcal L_0$ by $\mathcal L_0^{\otimes p^r}$ for some positive $r$, this divisor contains $C$, and the intersection properties of $\mathcal L_0$ imply that its generic fiber is contained in $D_{\Lambda_r}$. This implies the result.

In our situation, we do not know that $C$ parametrizes a family of $K3$ surfaces. However, the Kuga-Satake map provides us with a family of supersingular abelian varieties over $C$, which will be enough for our purposes. We now proceed with the proof.

\bigskip

By assumption, there exists a finite extension $R$ of $W$ with residue field $k$ and fraction field $F$, and a $R$-point $P$ of $\overline T$, such that the generic fiber of $P$ is a $F$-point $t$ of $\overline T$ lying on $D_{\Lambda}$ and such that the special fiber of $P$ is a $k$-point $t_0$ of $C$.

Let $\pi : \mathcal A \ra \overline{T}_R$ be the family of polarized abelian varieties obtained by pulling back the universal family over $\mathcal A_{g, d', n}$ by the Kuga-Satake map $\overline\kappa$ of Proposition \ref{finite}. Using Equation (\ref{sub}) of section \ref{KS}, there exists an endomorphism $\phi$ of the polarized abelian variety $\mathcal A_s$ and a component $D$ of the intersection of $D_{\Lambda}$ with $T_F$ such that, locally, $D$ is the image in $T_F$ of the versal deformation space of the pair $(\mathcal A_t, \phi)$. In other words, $D$ is the locus in $T_F$ where $\phi$ deforms with $\mathcal A_t$.

The point $t$ specializes to $t_0$. Let us write $\mathcal A_0$ for $\mathcal A_{t_0}$. The endomorphism $\phi$ specializes to an endomorphism $\phi_0$ of $\mathcal A_0$. There exists a positive integer $n$ such that the pair $(\mathcal A_0, p^r \phi_0)$ deforms over $C$. Indeed, the generic fiber $\mathcal A_{k(C)}$ of the universal polarized abelian scheme $\mathcal A$ over $C$ is supersingular by Proposition \ref{equal}, which implies that the cokernel of the specialization map
$$\mathrm{End}(\mathcal A_{k(C)})\ra \mathrm{End}(\mathcal A_0)$$
is killed by a power of $p$. In particular, the pair $(\mathcal A_0, p^r \phi_0)$ deforms to a pair $(\mathcal A_1, \psi_1)$ such that $\mathcal A_1$ is the Kuga-Satake abelian variety associated to the variety $X$ itself over $k$. Let $t_1$ be the corresponding point of $T$.

\bigskip

We want to study the versal deformation space of the pair $(\mathcal A_1, \psi_1)$ by making use of $X$. This can be done using the following canonical embedding
\begin{equation}\label{subcrys1}
H^2_{crys}(X/W)_{prim} \hookrightarrow \mathrm{End}(H^1_{crys}(\mathcal A_1/W))
\end{equation}
induced by Equation (\ref{subcrys}).
Given a lifting of the polarized variety $X$ to characteristic zero, the groups above are identified to relative de Rham cohomology groups, and (\ref{subcrys1}) is flat and respects the Hodge filtration. 

\begin{lem}\label{divisor}
Let $[\psi_1]\in \mathrm{End}(H^1_{crys}(\mathcal A_1/W))$ be the crystalline cohomology class of $\psi_1$. Then $[\psi_1]$ lies in the image of $H^2_{crys}(X_1/W)_{prim}$ by the morphism (\ref{subcrys1}).
\end{lem}

\begin{proof}
Since the cokernel of (\ref{subcrys1}) is torsion-free, we can work with crystalline cohomology groups tensored with $K$ and only show that $[\psi_1]$ lies in the image of $H^2_{crys}(X_1/K)$.

Let $\widehat B$ be the formal neighborhood of $t_0$ in $\overline T$. The formal scheme $\widehat B$ is flat over $\mathrm{Spf}(R)$ since $\overline T$ is flat over $\mathrm{Spec}(W)$. Recall the morphism of Equation (\ref{subcrys})
$$R^2\pi_*\Omega^{\bullet}_{\widehat{\mathcal X}/\widehat B}(1)_{prim} \hookrightarrow \mathrm{End}(R^1\psi_*\Omega^{\bullet}_{\widehat{\mathcal A}/\widehat B}).$$

Let $\alpha$ be the cohomology class of $\phi$ in $\mathrm{End}(H^1_{dR}(\mathcal A_t/F))$. By \cite[Proposition 8.9]{Ka70}, there exists a unique flat section $\widetilde \alpha$ of $\mathrm{End}\mathbb (R^1\pi_*\Omega^{\bullet}_{\mathcal A/\widehat B[1/p]})$ over $\widehat B[1/p]$ passing through $\alpha\,$\footnote{The result of \cite{Ka70} above only works over a smooth base, and $\widehat B[1/p]$ might not be smooth. However, the abelian scheme $\mathcal A$ over $\widehat B$ comes from the universal abelian scheme over $\mathcal A_{g, d', n}$, which is smooth and where Katz' result applies.}.

Since $\widetilde{\alpha}$ is flat and $\alpha$ belongs to $H^2_{dR}(\mathcal A_t/F)_{prim}$ by assumption, $\widetilde{\alpha}$ comes from a section of $R^2\pi_*\Omega^{\bullet}_{\widehat{\mathcal X}/\widehat B[1/p]}(1)_{prim}{[1/p]}$ over $\widehat B[1/p]$.

\bigskip

On the other hand, the endomorphism $\psi_1$ deforms by assumption  over $C$ to an endomorphism $\widetilde{\psi}$ of the abelian scheme $\mathcal A/C$. The crystalline cohomology class $[\widetilde\psi]$ of $\widetilde\psi$ induces a section of the convergent isocrystal $R^2\pi_*\Omega^{\bullet}_{\widehat{\mathcal X}/\widehat B}(1)_{prim}[1/p] _{|C}$. 

By definition of $\alpha$, we have $\widetilde{\alpha}_t=[\widetilde\psi]_{t_0}$ under the various identifications. By flatness, if $t'$ is a point of $\widehat B[1/p]$ that specializes to the generic point $\eta_C$ of $\widehat C$, we get $\widetilde{\alpha}_{t'}=[\widetilde\psi]_{\eta_C}$. By the remarks above, this implies the result.
\end{proof}

\bigskip

The preceding lemma allows us to control the deformation theory of the pair $(\mathcal A_1, \psi_1)$ as in the sketch of proof above. Indeed, standard deformation theory arguments and, for instance, Grothendieck-Messing theory as in \cite{Me72} -- see also \cite[Theorem 2.4]{CO09} for a summary --shows that the obstruction to deforming $\psi_1$ with $\mathcal A_1$ is controlled by the group $H^2(X, \mathcal O_{X})$. As this $k$-vector space is one-dimensional, the versal deformation space of $(\mathcal A_1, \psi_1)$ is a divisor $\Sigma$ in the formal neighborhood $\widehat T$ of $t_1$ in $T$. 

By the argument of \cite[Theorem 2.9]{Og79}, no component of $\Sigma$ can dominate the special fiber $T_k$. Since the result is not stated as such in \cite{Og79} as we do not know a priori that $\psi_1$ comes from a line bundle on $X_1$, let us sketch how Ogus' proof works in our setting. 

Ogus shows by a dimension count that $T_k$ contains an ordinary point. As a consequence, if some component of $\Sigma$ dominates $T_k$, then one can deform $(\mathcal A_1, \psi_1)$ to a pair $(\mathcal A_2, \psi_2)$ where $\mathcal A_2$ is the Kuga-Satake variety of an ordinary fiber of $\pi$. There we can assume that $\psi_2$ is not divisible by $p$, and the deformation arguments of \cite[Proposition 2.2]{Og79} give the result.

This shows that $\Sigma$ is a flat divisor in $\widehat{T}$. Furthermore, one of its component contains the supersingular component $C$ by assumption. As a consequence, the Zariski closure of the generic fiber $\Sigma_K$ of $\Sigma$ contains $C$.

We claim that $\Sigma_K$ is included in some $D_{\Lambda_r}$. First, let $t$ be any complex point of $\Sigma_K$. By definition, $t$ corresponds to a polarized variety $\mathcal X_t$ and an endomorphism $\psi$ of its Kuga-Satake abelian variety $\mathcal  A_t$. Furthermore, again by assumption, the cohomology class of $\psi$ lies in the subspace $H^2(X, \Z)_{prim}$ of $\mathrm{End}(H^1(A_t, \Z))$. 

Since the Kuga-Satake correspondence is induced by a Hodge class, the cohomology class $[\psi]$ of $\psi$ in Betti cohomology is a Hodge class. As a consequence, it comes from a Hodge class in $H^2(X, \Z)_{prim}$, which shows by the Lefschetz $(1,1)$ theorem that it is the cohomology class of a line bundle $\mathcal L$ on $X$. It is easy to check that the lattice generated by $c_1(\mathcal L)$ and the polarization in the N\'eron-Severi group of $X$ is isomorphic to $\Lambda_r$. This concludes the proof.
\bigskip

\end{proof}

\subsection{Finding one line bundle}

Let 
$$\Lambda_i=\left(\begin{matrix}d & a_i\\
                                          a_i  & b_i\end{matrix}\right)$$
be an infinite family of pairwise non-isomorphic lattices, and let $D_i=D_{\Lambda_i}$. By Theorem \ref{borcherds}, there exists a Cartier divisor $D$ in $T_K$ supported on a finite union of the $D_i$ such that $\mathcal O(D)=\lambda^{\otimes a}$ for some positive integer $a$. 

\begin{prop}\label{closure}
Let $\overline D$ be the Zariski closure of $D$ in $\overline T$. Then the special fiber $\overline D_k$ of $\overline D$ is an ample $\Q$-Cartier divisor.
\end{prop}

\begin{proof}
Consider the Kuga-Satake mapping 
$$KS_{\C} : \mathcal S_{n, \C}\ra \mathcal A_{g, d', n, \C}$$
of section \ref{period}.

By the argument of \cite[Proposition 5.8]{Mau12}, if $\lambda_{\mathcal A}$ denotes the Hodge bundle on $\mathcal A_{g, d', n, \C}$, there exists a positive integer $r$ such that 
$$KS_{\C}^*(\lambda_{\mathcal A}^{\otimes r})=\lambda^{\otimes(2^{b-1}r)},$$
where $b$ is the second Betti number of $X$.

As a consequence, after pulling back to $T_K$, we can write 
$$\mathcal O(MD)=\kappa_K^*(\lambda_{\mathcal A}^N)$$
for some positive integers $M$ and $N$ -- taking powers of line bundles to descend the equality from $T_{\C}$ to $T_K$.

\bigskip

Let $U$ be the smooth locus of $\overline T \ra \mathrm{Spec}\,W$, and let $D'$ be the closure of $D$ in $U$. By Lemma 5.12 in \cite{Mau12}, $\mathcal O(MD')=\overline{\kappa}^*(\lambda_{\mathcal A}^N)_{|U}$.
Let $s$ be a section of $\overline{\kappa}^*(\lambda_{\mathcal A}^c)$ over $U$ with divisor $bD'$. 

Since $\overline T$ is normal, the complement of $U$ in $\overline T$ has codimension at most $2$ in $\overline T$ and $s$ extends to a section of $\overline{\kappa}^*(\lambda_{\mathcal A}^N)$ over $\overline T$. The divisor of $s$ is the closure of $MD'$ in $\overline T$, which precisely means that the divisor of $s$ in $\overline T$ is $M\overline D$ and shows that 
$$\mathcal O(M\overline D)=\overline{\kappa}^*(\lambda_{\mathcal A}^N).$$

By \cite[V.2.3]{FC90}, the Hodge line bundle $\lambda_{\mathcal A}$ is ample on $\mathcal A_{g, d', n, k}$. Since $\overline{\kappa}$ is finite, this proves that $M\overline D_k$ is an ample Cartier divisor and concludes the proof.
\end{proof}

We now prove the following first step in the direction of Theorem \ref{weak}. We keep the notations as above.

\begin{prop}\label{one}
There exists a complex point $t$ of $T$ such that $\mathcal X_t$ specializes to $X$ and has Picard number at least $2$.
\end{prop}

\begin{proof}
Let $x$ be the $k$-point of $T$ corresponding to $X$. By \cite[Theorem 15]{Og11}, the supersingular locus in $T$ is closed of dimension $s=b-3-E((b-1)/2)$, where $E$ is the integer part function. Indeed, the height of varieties parametrized by $\mathcal X\ra T$ varies between $1$ and $E((b-1)/2)$ if it is finite, and the locus of points in $T$ with height at least $h$ has codimension $h-1$. Together with Proposition \ref{proper}, this shows that there is a nontrivial proper $s$-dimensional component $C$ in the supersingular locus of  $\overline T$ containing $x$. Note that $s>0$.

Let $\overline D$ be as in the preceding section, with $B_i=0$. By Proposition \ref{closure}, some multiple of $\overline D_k$ is an ample Cartier divisor. As a consequence, the intersection of $\overline D$ and $C$ is not empty. By Proposition \ref{inclus}, there is a lattice $\Lambda$ of the form
$$\Lambda_i=\left(\begin{matrix}2d & a\\
                                          a  & 0\end{matrix}\right)$$
such that the Zariski-closure of $D_{\Lambda}$ in $\overline{T}$ contains $C$. As a consequence, $X$ is the specialization of a variety with Picard number at least $2$.
\end{proof}

\noindent\textbf{Remark.} In the case of $K3$ surfaces, this argument shows that $X$ carries an elliptic pencil and satisfies the Tate conjecture by Artin's theorem in \cite{Ar74}.

\section{Proof of Theorem \ref{weak}}\label{gen}

We now adapt the techniques of the preceding section to prove Theorem \ref{weak}. We keep the notations as above.

\subsection{Lifting many line bundles to characteristic zero}

In this section, we prove the following result.

\begin{prop}\label{induction}
Let $x$ be the point of $\overline T$ corresponding to $X$, and let $C$ be the connected component of the supersingular locus of $\overline T$ containing $x$. There exist a $k$-point $y$ of $C$ and a complex point $z$ of $\overline T$ with the following properties.
\begin{enumerate}
\item Under the identifications of Proposition \ref{finite}, the point $z$ specializes to $y$.
\item The weight $2$ Hodge structure parametrized by $z$ has Picard rank $b-3$.
\end{enumerate}
\end{prop}

We start with a generalization of Proposition \ref{inclus} to higher Picard numbers which might be of independent interest. Before stating it, let us introduce some notations.

Let $\Lambda$ be a nondegenerate lattice containing a primitive element $h$ of square $d$. We denote by $Z_{\Lambda}$ the locus in $\mathcal S_{n, \C}$ of points $s$ such that if $H_s$ is the weight $2$ Hodge structure on $L$ corresponding to $s$, there exists an embedding of $\Lambda$ in the N\'eron-Severi group of $H_s$ mapping $h$ to the class of the polarization. In case the rank of $\Lambda$ is $2$, we recover the divisor $D_{\Lambda}$ we used above. As before, $Z_{\Lambda}$ is defined over $\Q$.

\begin{prop}\label{lift10}
Let $x$ be the point of $\overline T$ corresponding to $X$ and let $C$ be a component of the supersingular locus in $\overline T$ passing through $x$. Then there exists a lattice $\Lambda$ of rank $E((b-1)/2)$ such that the Zariski-closure of $Z_{\Lambda}$ in $\overline T$ contains $C$.
\end{prop}

\begin{cor}
The variety $X$ admits a lift to a polarized variety of Picard rank at least $E((b-1)/2)$ in characteristic $0$.
\end{cor}

\begin{proof}[Proof of Proposition \ref{lift10}]
We prove by induction on $n\leq E((b-1)/2)$ that there exists a lattice $\Lambda$ of rank $n$ such that the Zariski closure of $Z_{\Lambda}$ in $\overline T$ contains $S$. We will argue as in Proposition \ref{inclus}, which deals with the rank $2$ case.

\bigskip

Let $n< E((b-1)/2)$ be a positive integer, and assume that there exists a lattice $\Lambda$ of rank $n$ such that the Zariski-closure of $Z_{\Lambda}$ in $\overline T$ contains $C$. Let us first remark that $Z_{\Lambda}$ is itself a Shimura subvariety of $\mathcal S_{n}$, associated to the orthogonal of a copy of $\Lambda$ in the lattice $L$. As such, it is a Shimura variety of orthogonal type corresponding to a lattice of signature $(2, b-2-n)$. 

The Noether-Lefschetz locus on $Z_{\Lambda}$ is a countable union of divisors satisfying the following analog of Theorem \ref{borcherds}. Let $\Lambda'$ be a nondegenerate rank $n+1$ lattice containing $\Lambda$. The variety $Z_{\Lambda'}$ is naturally a divisor in $Z_{\Lambda}$. The proof of \cite[Theorem 3.1]{Mau12} translates immediately to show that the analog of Theorem \ref{borcherds} holds for the divisors $Z_{\Lambda'}$ in $Z_{\Lambda}$.

We can now use the ampleness arguments of Proposition \ref{closure}, working this time with the Zariski closure $\overline{Z}_{\Lambda}$ of $Z_{\Lambda}$ in $\overline T$, to show that there exists a lattice $\Lambda'$ of rank $n+1$ containing $\Lambda$ such that the intersection of $\overline{Z}_{\Lambda'}$ and $C$ is not empty\footnote{The only difference with Proposition \ref{closure} is that $\overline{Z}_{\Lambda}$ is not normal a priori. However, one can work on the normalization of $\overline{Z}_{\Lambda}$ and carry on with the proof.}.

\bigskip

At this point, we can repeat the proof of Proposition \ref{inclus} to show that some $\overline{Z}_{\Lambda'}$ actually contains the support of $C$. The only result that does not go through is the following. Let $\mathcal X_t$ be a fiber of $\pi$ over a $k$-point $t$ of $T$ with an embedding of $\Lambda$ in $\mathrm{Pic(X_t)}$, and let $(\mathcal X_t, S)$ be an irreducible component of a versal $k$-deformation of the pair $(X_t, \Lambda)$ in $T$. We need to show that the geometric generic fiber $\mathcal X_{\overline \eta}$ is ordinary with Picard group of rank $n$. This is a generalization of \cite[Theorem 2.9]{Og79}.

First remark that by standard deformation theory, the dimension of $\overline{Z}_{\Lambda}$ is at least $b-2-n$. On the other hand, Ogus shows in \cite[Theorem 2.9]{Og79} that the dimension of the non-ordinary locus in $S$ is at most $\mathrm{max}(n, b-3-n)$. Since $n<E((b-1)/2)$, this shows that $\mathcal X_{\overline \eta}$ is ordinary. Usual deformation theory of ordinary $K3$ crystals allows us to conclude that $S$ is of dimension $b-2-n$ and that the conclusion holds.
\end{proof}

Using the result above, we can prove Proposition \ref{induction}. Let us show by induction on $n\leq b-3-E((b-1)/2)$ that there exist a nondegenerate lattice $\Lambda$ of rank $E((b-1)/2)+n$ such that the intersection of $\overline{Z}_{\Lambda}$ with $C$ is a non-empty subscheme $C_{\Lambda}$ of $C$ of dimension at least $b-3-E((b-1)/2)-n$. For $n=E((b-1)/2)$, this gives the conclusion of Proposition \ref{induction}. 

For $n=0$, this is the statement of Proposition \ref{lift10}, since $C$ is of dimension $b-3-E((b-1)/2)$. Assume that the result we just stated holds for some $n< b-3-E((b-1)/2)$. As in the proof of Proposition \ref{lift10} above, since the dimension of $C_{\Lambda}$ is positive, we can find a nondegenerate lattice $\Lambda'$ of rank $n+1$ containing $\Lambda$ such that $\overline{Z}_{\Lambda'}$ has non-empty intersection $C_{\Lambda'}$ with $C_{\Lambda}$. Since the dimension of $C_{\Lambda}$ is at least $b-3-E((b-1)/2)-n$, the dimension of $C_{\Lambda'}$ is at least $b-3-E((b-1)/2)-(n+1)$. This concludes the proof of Proposition \ref{induction}. \qed

\subsection{From Picard rank $b-3$ to Picard rank $b$}

In this section, we show how to derive Theorem \ref{weak} from Proposition \ref{induction}.

\begin{proof}[Proof of Theorem \ref{weak}]
We start with a Hodge-theoretic lemma.

\begin{lem}\label{Morrison}
Let $H$ be a weight $2$ polarized Hodge structure with $h^{2, 0}=1$. Assume that the codimension of the space of Hodge classes in $H$ is at most $3$. Let $A$ be the Kuga-Satake variety of $H$ together with a polarization, and let 
$$p : \mathrm{End}(H^1(A, \Q))\ra \mathrm{End}(H^1(A, \Q))$$
be the orthogonal projector onto $H$. Then $p$ is induced by an algebraic correspondence of $(A\times A)^2$.
\end{lem}

\begin{proof}
Standard computations show that the Kuga-Satake variety of $H$ is isogenous -- in a functorial way -- to a power of the Kuga-Satake variety associated to the transcendental lattice of $H$. As a consequence, we can assume that the dimension of $H$ is $3$. 

In that case, $A$ is an abelian variety of dimension $2^{3-1}=4$. However, we know that $A$ is isogenous to the square of the Kuga-Satake variety obtained by considering the even Clifford algebra of $H$. It follows that $A$ is isogenous to the square of an abelian surface. In particular, $(A\times A)^2$ is isomorphic to a product of abelian surfaces and satisfies the Hodge conjecture by the main result of \cite{RM08}. This proves the theorem, as the projector $p$ is indeed given by a Hodge class.
\end{proof}

We now use the notations of Proposition \ref{induction}, and we want to prove that the Picard rank of $X$ is $b$. Let $A_z$ be the Kuga-Satake variety of the weight $2$ Hodge structure $H_z$ parametrized by $z$. By assumption, $A_z$ has good, supersingular reduction at $p$. Let $A_y$ be this smooth reduction. By the preceding lemma, we have an algebraic correspondence of codimension $2$ on $(A_z\times A_z)^2$ which acts as the orthogonal projector 
$$p : \mathrm{End}(H^1(A_z, \Q))\ra \mathrm{End}(H^1(A_z, \Q))$$
onto $H_z$.

The correspondence $p$ specializes to a correspondence $p_y$ on $(A_y\times A_y)^2$. By the smooth base-change theorem, the image of $p_y$ acting on crystalline cohomology is $(b-1)$-dimensional. Furthermore, since $A_y$ is supersingular, its crystalline cohomology is spanned by algebraic cycles. In particular, we can find endomorphisms $\psi_1, \ldots, \psi_{b-1}$ of $A_y$ that lie in the image of $p$ and span the image of $p$.

Arguing as in Proposition \ref{inclus} and Lemma \ref{divisor}, it is easy to see that after replacing the $\psi_i$ by some nonzero multiples, each of the $\psi_i$ deform over $C$, and that they lift to characteristic zero. A Hodge-theoretic argument as in the end of Proposition \ref{inclus} then allows us to conclude that the $\psi_i$ deform to classes of line bundles on $X$ spanning the primitive cohomology of $X$. This concludes the proof of Theorem \ref{weak}.
\end{proof}

\bibliography{bibK3}{}
\bibliographystyle{plain}

\end{document}